\documentclass[12pt]{article}
\usepackage{amssymb}
\usepackage{amsmath}
\usepackage{amsthm}

\newtheorem{thm}{Theorem}[section]
\newtheorem{cor}[thm]{Corollary}

\def\R{{\mathbb R}}
\def\E{{\mathbb E}}

\def\P{{\mathbb P}}

\def\NN{{\mathcal N}}
\def\1{{\mathbf 1}}
\def\a0{{\aleph_0}}

\def\0{{\mathbf{0}}}

\def\eps{{\varepsilon}}

\newcommand{\Lra}{{\Longleftrightarrow}}

\newcommand{\ra}{{\rightarrow}}

\newcommand{\cov}{{\rm cov}}
\newcommand{\Var}{{\rm Var}}

\def\is#1#2{\left\langle #1 , #2 \right\rangle}

\newcommand{\by}{{\bar{y}}}

\newcommand{\bz}{{\bar{z}}}

\title{The square negative correlation property for generalized Orlicz balls}

\author{Jakub Onufry Wojtaszczyk
\\
Department of Mathematics, Computer Science and Mechanics \\ University of Warsaw\\ ul. Banacha 2,
02-097 Warsaw, Poland\\ email: onufry@duch.mimuw.edu.pl}

\date{September 11, 2005}

\begin{document}

\maketitle
\begin{abstract}
Recently Antilla, Ball and Perissinaki proved that the squares of coordinate 
functions in $l_p^n$ are negatively correlated. This paper extends
their results to balls in generalized Orlicz norms on $\R^n$. From this, the concentration
 of the Euclidean
norm and a form of the Central Limit Theorem for the generalized Orlicz balls
is deduced. Also, a counterexample for the square negative correlation hypothesis 
for 1-symmetric bodies is given. 
\end{abstract}

\section{Introduction}
Given a convex, central-symmetric body $K\subset \R^n$ of volume 1, consider the random variable 
$X = (X_1, X_2, \ldots, X_n)$, uniformly distributed on $K$. We are interested in 
determining whether the vector has the {\em square negative correlation}, i.e. if $$\cov(X_i^2,X_j^2) := \E (X_i^2 X_j^2) - \E X_i^2 \E X_j^2 \leq 0.$$ We assume that $K$ is in {\em isotropic 
position}, i.e. that $$\E X_i = 0 \hbox{\ \ \ and \ \ \ } \E X_i \cdot X_j = L_K^2\delta_{ij},$$
where $\delta_{ij}$ is the Kronecker delta and $L_K$ is a positive constant. Since any convex body not supported on an affine subspace has an affine image which is in isotropic position, this is not a restrictive assumption.

The motivation in studying this problem comes from the so-called central limit problem for convex bodies, which is to show that most of the one-dimensional projections of the uniform measure on a convex body are approximately normal. It turns out that the bounds on the square correlation can be crucial to estimating the distance between the one-dimensional projections and the normal distribution (see for instance \cite{ABP}, \cite{meckes}). A related problem is to provide bounds for the quantity $\sigma_K$, 
defined by $$\sigma_K^2 = \frac{\Var (|X|^2)}{nL_K^4} = \frac{n\Var (|X|^2)}{(\E |X|^2)^2},$$ where $X$ is uniformly distributed on $K$.
It is conjectured (see for instance \cite{BK}) that $\sigma_K$ is bounded by a universal constant 
for any convex symmetric isotropic body. Recently Antilla, Ball and Perissinaki (see \cite{ABP}) 
observed that for $K = l_p^n$ the covariances of $X_i^2$ and $X_j^2$ are negative for $i \neq j$,
and from this deduced a bound on $\sigma_K$ in this class. 

In this paper we shall study the covariances of $X_i^2$ and $X_j^2$ (or, more generally, of any
functions depending on a single variable) on a convex, symmetric and isotropic body. We will show a general formula to calculate the covariance for given functions and $K$, and from this formula deduce the covariance of any increasing functions of different variables, in particular of the functions $X_i^2$ and $X_j^2$, has to be negative on generalized Orlicz balls. Then we follow \cite{ABP} to arrive at a concentration property and \cite{meckes} to get a Central Limit Theorem variant for generalized Orlicz balls.

The layout of this paper is as follows. First we define notations which will be used throughout the paper. In Section \ref{gener} we transform the formula for the square correlation into a 
form which will be used further on. In Section \ref{orlic} we use the formula and the Brunn-Minkowski inequality to arrive at the square negative correlation property for generalized Orlicz balls. In Section \ref{conse} we show the corollaries, in particular a central-limit theorem for
generalized Orlicz balls. Section \ref{count} contains another application of the formula from Section \ref{gener}, a simple counterexample for the square negative correlation hypothesis for 1-symmetric bodies.

\paragraph{Notation}

Throughout the paper $K \subset \R^n$ will be a convex central-symmetric body of volume 1 in isotropic position. Recall that by isotropic position we mean that for any vector $\theta \in S^{n-1}$ we have $\int_K \is{\theta}{x}^2 dx = L_K^2$ for some constant $L_K$. 
For $A \subset \R^n$ by $|A|$ we will denote the Lebesgue volume of $A$. 
For $x \in \R^n$, $|x|$ will mean the Euclidean norm of $x$. We assume that 
$\R^n$ is equipped with the standard Euclidean structure and with the 
canonic orthonormal base $(e_1,\ldots,e_n)$. For $x \in \R^n$ by $x_i$ we
shall denote the $i$th coordinate of $x$, i.e. $\is{e_i}{x}$. We will
consider $K$ as a probability space with the Lebesgue measure restricted to $K$ as
the probability measure. If there is any danger of confusion, then $\mathbb{P}_K$ will denote the probability with respect to this measure, $\E_K$ will denote the expected value with respect to $\mathbb{P}_K$, and so on. By $X$ we will usually denote the $n$-dimensional random vector
equidistributed on $K$, while $X_i$ will denote its $i$th coordinate. By the covariance
$\cov(Y,Z)$ for real random variables $Y$, $Z$ we mean $\E (Y Z) - \E Y \E Z$. By an 1-symmetric
body $K$ we mean one that is invariant under reflections in the coordinate hyperplanes, or equivalently, such a body that $(x_1,x_2,\ldots,x_n) \in X \Lra (\eps_1 x_1,\eps_2 x_2,\ldots,\eps_n x_n \in X)$ for any choice of $\eps_i \in \{-1,1\}$. The parameter $\sigma_K$, as in \cite{BK}, will be defined by $$\sigma_K^2 = \frac{\Var (|X|^2)}{nL_K^4} = \frac{n\Var (|X|^2)}{(\E |X|^2)^2}.$$

For any $n\geq 1$ and convex increasing functions $f_i : [0,\infty) \ra [0, \infty)$, $i = 1,\ldots, n$ satisfying $f_i(0) = 0$ (called the Young functions) we define the generalized Orlicz ball $K \subset \R^n$ to be the set of points $x = (x_1,\ldots,x_n)$ satisfying $$\sum_{i=1}^n f_i(|x_i|) \leq 1.$$ This is easily proven to be convex, symmetric and bounded, thus 
$$\|x\| = \inf \{ \lambda : x \in \lambda K\}$$ defines a norm on $\R^n$. In the case of equal
functions $f_i$ the norm is called an {\em Orlicz norm}, in the general case a {\em generalized Orlicz} norm. Examples of Orlicz norms include the $l_p$ norms for any $p \geq 1$ with $f(t) = |t|^p$ being the Young functions. The generalized Orlicz spaces are also referred to as modular sequence spaces (I thank the referee for pointing this out to me).

\section{The general formula}\label{gener}
We wish to calculate $\cov (f(X_i), g(X_j))$, where $f$ and $g$ are univariate functions, $i \neq j$ and $X_i,X_j$ are the coordinates of the random vector $X$, equidistributed on a convex, symmetric and isotropic body $K$. For simplicity we will assume $i = 1$, $j = 2$ and denote $X_1$ by $Y$ and
$X_2$ by $Z$. For any $(y,z) \in \R^2$ let $m(y,z)$ be equal to the $n-2$-dimensional Lebesgue
measure of the set $(\{(y,z)\}\times \R^{n-2}) \cap K$. We set out to prove:
\begin{thm} \label{general} For any symmetric, convex body $K$ in isotropic position and any functions $f$, $g$ we have $$\cov(f(Y),g(Z)) = \int_{\R^4, |y| > |\by|, |z|>|\bz|}
\big(m(y,z)m(\by,\bz) - m(y,\bz)m(\by,z)\big) \big(f(y) - f(\by)\big)\big(g(z) -
g(\bz)\big).$$\end{thm}

Furthermore, for 1-symmetric bodies and symmetric functions we will have the following corrolary:
\begin{cor} \label{uncond} For any symmetric, convex, uncondtitional body $K$ in isotropic position and symmetric functions $f$, $g$ we have $$\cov(f(Y),g(Z)) =16 \  \int_{\R^4, y > \by > 0, z>\bz > 0}
\big(m(y,z)m(\by,\bz) - m(y,\bz)m(\by,z)\big)\big(f(y) - f(\by)\big)\big(g(z) -
g(\bz)\big).$$\end{cor}

The corollary is a simple consequence of the fact that for symmetric functions $f$ and $g$ and an 1-symmetric body $K$ the integrand is invariant under the change of the sign of any of the variables, so we may assume all of them are positive.

As concerns the sign of $\cov(f,g)$, which is what we set out to determine, we have the following simple corollary:
\begin{cor} \label{signuncond} For any central-symmetric, convex, 1-symmetric body $K$ in isotropic position and symmetric functions $f$, $g$ that are non-decreasing on $[0,\infty)$ if for all $y > \by > 0$, $z > \bz > 0$ we have \begin{equation}\label{crossmass}m(y,\bz) m(\by,z) \geq m(y,z) m(\by,\bz),\end{equation} then $$\cov(f,g) \leq 0.$$ Similarly, if the opposite inequality is satisfied for all $y > \by >0 $ and $z > \bz > 0$, then the covariance is non-negative.\end{cor}

\begin{proof}The second and third bracket of the integrand in Corollary \ref{uncond} is positive under the assumptions of Corollary \ref{signuncond}. Thus if we assume the first bracket is negative, then the whole integrand is negative, which implies the integral is negative, and vice-versa.\end{proof}

\begin{proof}[Proof of Theorem \ref{general}] We have $$\cov(f(Y), g(Z)) = \E f(Y) g(Z) - \E f(Y) \E g(Z).$$ From the Fubini theorem we have 
$$\E f(Y) g(Z) = \int_{R^2} m(y,z) f(y) g(z),$$ and similar equations for $\E f(Y)$ and $\E g(Z)$. 

For any function $h$ of two variables $a,b \in A$ we can write $\int_{A^2} h(a,b) = \int_{A^2} h(b,a) = \frac{1}{2} \int_{A^2} h(a,b) + h(b,a)$. We shall repeatedly use this trick to transform the formula
for the covariance of $f$ and $g$ into the required form:

\begin{eqnarray*}
\E f(Y) \E g(Z) &=& \int_{\R^2} m(y,z) f(y)  \int_{\R^2} m(\by,\bz) g(\bz)  \\ &=&
\int_{\R^4} m(y,z) m(\by,\bz) f(y)
g(\bz) = \int_{\R^4} m(\by,\bz) m(y,z) f(\by) g(z) =\\ &=&
\frac{1}{2} \int_{\R^4} m(\by,\bz) m(y,z) \big(f(\by) g(z) + f(y)
g(\bz)\big).\end{eqnarray*}

We repeat this trick, exchanging $z$ and $\bz$ (and leaving $y$ and $\by$ unchanged):

\begin{eqnarray*}
\E f(Y) \E g(Z) &=& \frac{1}{4}
\int_{\R^4} m(\by,\bz) m(y,z) \big(f(y) g(\bz) + f(\by) g(z)\big) +
m(\by,z) m(y,\bz) \big(f(y) g(z)+ f(\by) g(\bz)\big).\end{eqnarray*}

We perform the same operations on the second part of the covariance. To get a integral over
$\R^4$ we multiply by an $\E 1$ factor (this in effect will free us from the assumption that
the body's volume is 1):
\begin{eqnarray*}
\E f(Y) g(Z) \E 1 &=& \int_{\R^4} m(y,z) m(\by,\bz) f(y)
g(z) \\ &=& \frac{1}{4} \int_{\R^4} m(y,z)
m(\by,\bz) \big(f(y)g(z) + f(\by)g(\bz)\big) + m(y,\bz)m(\by,z)
\big(f(y)g(\bz) + f(\by)g(z)\big).\end{eqnarray*}

Thus:

\begin{eqnarray*} &&\cov (f(Y),g(Z)) = \E (f(Y) g(Z)) \E 1 - \E f(Y) \E g(Z)
=\\&=& \frac{1}{4} \bigg(\int_{\R^4} m(y,z) m(\by,\bz) \big(f(y)g(z) + f(\by)g(\bz)\big) +
m(y,\bz)m(\by,z) \big(f(y)g(\bz) + f(\by)g(z)\big)
-\\&&-
m(\by,\bz) m(y,z) \big(f(y)
g(\bz) + f(\by) g(z)\big) - m(\by,z) m(y,\bz) \big(f(y) g(z)+ f(\by)
g(\bz)\big)\bigg)
=\\&=&
\frac{1}{4}\int_{\R^4} \bigg(\big(m(y,\bz)m(\by,z) - m(y,z)m(\by,\bz)\big) \big(f(y)
g(\bz) + f(\by) g(z)\big) +\\&&+ \big(m(y,z)m(\by,\bz) -
m(\by,z) m(y,\bz)\big) \big(f(y)g(z) + f(\by)g(\bz)\big)\bigg) =\\&=&
\frac{1}{4}\int_{\R^4} \big(m(y,\bz)m(\by,z) - m(y,z)m(\by,\bz)\big) \big(f(y)
g(\bz) + f(\by) g(z) - f(y)g(z) - f(\by)g(\bz)\big) =\\&=&
\frac{1}{4}\int_{\R^4} \big(m(y,\bz)m(\by,z) - m(y,z)m(\by,\bz)\big) \big(f(y)
- f(\by)\big)\big(g(\bz) - g(z)\big)
\end{eqnarray*}

Finally, notice that if we exchange $y$ and $\by$ in the above formula, then 
the formula's value will not change --- the first and second bracket will change signs, and
the third will remain unchanged. The same applies to exchanging $z$ and $\bz$. Thus
$$\cov(f,g) = \int_{\R^4, |y| > |\by|, |z|>|\bz|}
\big(m(y,z)m(\by,\bz) - m(y,\bz)m(\by,z)\big) \big(f(y) - f(\by)\big)\big(g(z) -
g(\bz)\big).$$

\end{proof}

\section{Generalized Orlicz spaces}\label{orlic}

Now we will concentrate on the case of symmetric, non-decreasing functions on generalized Orlicz spaces. We will prove the inequality (\ref{crossmass}):

\begin{thm} \label{gorb} If $K$ is a ball in an generalized Orlicz norm on $\R^n$, then for any $y > \by > 0$ and $z > \bz > 0$ we have \begin{equation}m(y,\bz) m(\by,z) \geq m(y,z) m(\by,\bz).\label{eq1} \end{equation}\end{thm}

From this Theorem and Corollary \ref{signuncond} we get
\begin{cor} \label{negcov} If $K$ is a ball in an generalized Orlicz norm on $\R^n$ and $f, g$ are symmetric functions that are non-decreasing on $[0,\infty)$, then $\cov_K(f,g) \leq 0$.\end{cor}

It now remains to prove the inequality (\ref{eq1}).
\begin{proof}[Proof of Theorem \ref{gorb}] Let $f_i$ denote the Young functions of $K$. Let us consider the ball $K' \subset \R^{n-1}$, being an generalized Orlicz ball defined by the Young functions $\Phi_1, \Phi_2, \ldots, \Phi_{n-1}$, where $\Phi_i(t) = f_{i+1}(t)$ for $i > 1$ and $\Phi_1(t) = t$ --- that is, we replace the first two Young functions of $K$ by a single identity function.

For any $x\in \R$ let $P_x$ be the set $(\{x\} \times \R^{n-2}) \cap K'$, and $|P_x|$ be its $n-2$-dimensional
Lebesgue measure. $K'$ is a convex set, thus, by the Brunn-Minkowski inequality (see for instance \cite{ff}) the function $x \mapsto |P_x|$ is a logarithmically concave function. This means that $x \mapsto \log |P_x|$ is a concave function, or equivalently that $$|P_{tx + (1-t)y}| \geq |P_x|^t \cdot |P_y|^{1-t}.$$ In particular, for given real positive numbers $a$, $b$, $c$ we have 
$$|P_{a+c}| \geq |P_a|^{b \slash (b+c)} |P_{a+b+c}|^{c \slash
(b+c)},$$
$$|P_{a+b}| \geq |P_a|^{c \slash (b+c)} |P_{a+b+c}|^{b \slash
(b+c)},$$ and as a consequence when we multiply the two inequalities, 
\begin{equation}|P_{a+b}| \cdot |P_{a+c}| \geq |P_a| \cdot |P_{a+b+c}|.\label{ff}\end{equation}

Now let us consider the ball $K$. Let us take any $y > \by > 0$ and $z > \bz > 0$. Let 
$a = f_1(\by) + f_2(\bz)$, $b = f_1(y) - f_1(\by)$, and
$c = f_2(z) - f_2(\bz)$. The numbers $a$, $b$ and $c$ are positive from the assumptions
on $y$, $z$, $\by$ and $\bz$ and because the Young functions are increasing. Then $m(\by,\bz)$ is equal to the measure of the set 
$$\{x_3,x_4,\ldots,x_n : f_1(\by) + f_2(\bz) + \sum_{i=3}^n f_i(x_i) \leq 1\} = \{x_3,x_4,\ldots,x_n : a + \sum_{i=2}^n \Phi_i(x_i) \leq 1\} = P_a.$$
Similarly $m(y,\bz) = |P_{a+b}|$, $m(\by,z) = |P_{a+c}|$ i
$m(y,z) = |P_{a+b+c}|$.  

Substituting those values into the inequality (\ref{ff}) we get the thesis: $$m(y,\bz) m(\by,z) \geq m(y,z) m(\by,\bz).$$\end{proof}

\section{The consequences}\label{conse}

For the consequences we will take $f(t) = g(t) = t^2$. The first simple consequence is the concentration property for generalized Orlicz balls. Here, we follow the argument of \cite{ABP} for $l_p$ balls. 

\begin{thm} For every generalized Orlicz ball $K \subset \R^n$ we have $$\sigma_K \leq \sqrt{5}.$$\end{thm}

\begin{proof} From the Cauchy-Schwartz inequality we have
$$ n^2L_K^4 = \bigg(\sum_{i=1}^n \E_K X_i^2 \bigg)^2 = \bigg(\E_K |X|^2 \bigg)^2  \leq \E_K |X|^4.$$ On the other hand from Corollary \ref{negcov} we have
\begin{eqnarray*} \E_K|X|^4 &=& \E_K \bigg(\sum_{i=1}^nX_i^2\bigg)^2 = \sum_{i=1}^n \E_K X_i^4 + \sum_{i\neq j} \E_K X_i^2 X_j^2  \\ &\leq& \sum_{i=1}^n \E_K X_i^4 + \sum_{i\neq j} \E_K X_i^2  \E_K X_j^2 \\ &=& \sum_{i=1}^n \E_K X_i^4 + n(n-1)L_K^4 .\end{eqnarray*}

As for 1-symmetric bodies the density of $X_i$ is symmetric and log-concave, we know (see e.g. \cite{klo}, Section 2, Remark 5) $$\E_K X_i^4 \leq 6 \bigg(\E_K X_i^2\bigg)^2 = 6 L_K^4,$$ thence $$n^2L_K^4 \leq \E_K |X|^4 \leq (n^2 + 5n)L_K^4.$$ This gives us $$\Var(|X|^2) = \E_K|X|^4 - n^2 L_K^4 \leq 5n L_K^4,$$
and thus $$\sigma_K^2 = \frac{\Var|X|^2}{nL_K^4} \leq 5.$$\end{proof}

\begin{cor} For every generalized Orlicz ball $K \subset \R^n$ and for every $t > 0$ we have $$\mathbb{P}_K\bigg(\bigg|\frac{|X|^2}{n} - L_K^2\bigg| \geq t\bigg) \leq \frac{5 L_K^4}{nt^2}$$ and $$\mathbb{P}_K\bigg(\bigg|\frac{|X|}{\sqrt{n}} - L_K\bigg| \geq t\bigg) \leq \frac{5 L_K^2}{nt^2}$$\label{corcor} \end{cor}

\begin{proof} From the estimate on the variance of $|X|^2$ and Chebyshev's inequality we get $$t^2\mathbb{P}_K\bigg(\bigg|\frac{|X|^2}{n} - L_K^2\bigg| \geq t\bigg) \leq \E_K\bigg(\frac{|X|^2}{n} - L_K^2\bigg)^2 \leq \frac{1}{n^2} \Var(|X|^2) \leq \frac{5}{n} L_K^4.$$

For the second part let $t > 0$. We have \begin{eqnarray*} \mathbb{P}_K (|X| - \sqrt{n}L_K| \geq t\sqrt{n}) &\leq& \mathbb{P}_K(|X|^2 - nL_K^2| \geq tnL_K) \\ &\leq &\frac{5L_K^4}{t^2 n L_K^2} = \frac{5L_K^2}{t^2n}.\end{eqnarray*}\end{proof}

This result confirms the so-called {\em concentration hypothesis} for generalized Orlicz balls. The hypothesis, see e.g. \cite{BK}, states that the Euclidean norm concentrates near the value $\sqrt{n}L_K$ as a function on $K$. More precisely, for a given $\eps > 0$ we say that $K$ satisfies the {\em $\eps$-concentration hypothesis} if  $$\P_K \bigg(\bigg|\frac{|X|}{\sqrt{n}} - L_K\bigg | \geq \eps L_K\bigg) \leq \eps.$$ From Corollary \ref{corcor} we get that the class of generalized Orlicz balls satisfies the $\eps$-concentration hypothesis with $\eps = \sqrt{5} n^{-1 \slash 3}$.

A more complex consequence is the Central Limit Property for generalized Orlicz balls. For $\theta \in S^{n-1}$ let $g_\theta(t)$ be the density of the random variable $\is{X}{\theta}$. Let $g$ be the density of $\NN(0,L_K^2)$. Then for most $\theta$ the density $g_\theta$ is very close to $g$. More precisely, by part 2 of Corollary 4 in \cite{meckes} we get 
\begin{cor} There exists an absolute constant $c$ such that $$\sup_{t\in \R} \bigg|\int_{-\infty}^t \big(g_\theta(s) - g(s) \big)ds\bigg| \leq c \|\theta\|_3^{3\slash 2}.$$ \end{cor}

\section{The counterexample for 1-symmetric bodies}\label{count}
It is generally known that the negative square correlation hypothesis does not hold in general in the class of 1-symmetric bodies. However, the formula from section \ref{gener} allows us to give a counterexample without any tedious calculations. Let $K \subset \R^3$ be the ball of the norm defined by $$\|(x,y,z)\| = |x| + \max \{|y|,|z|\}.$$ The quantity $m(y,z)$ considered in Corollary \ref{signuncond}, defined as the volume of the cross-section $(\R \times \{y,z\}) \cap K$ is equal to $2 (1 - \max\{|y|,|z|\})$ for $|y|,|z| \leq 1$ and $0$ for greater $|y|$ or $|z|$. To check the inequality (\ref{crossmass}) for $y > \by > 0$ and $z > \bz > 0$ we may assume without loss of generality that $y \geq z$ (as $K$ is invariant under the exchange of $y$ and $z$). 
We have
\begin{eqnarray*} m(y,\bz) m(\by,z) &-& m(y,z) m(\by,\bz) =\\ &=&4 (1- \max\{y,\bz\})(1-\max\{\by,z\}) - 4(1-\max\{y,z\})(1-\max\{\by,\bz\}) \\ &=& 4 (1 - y) (1-\max\{\by,z\}) - 4 (1-y) (1-\max\{\by,\bz\}) \\&=& 4 (1 - y) (\max\{\by,\bz\} - \max\{\by,z\}).\end{eqnarray*}
As $y \leq 1$ all we have to consider is the sign of the third bracket. However, as $z > \bz$, the third bracket is never positive, and is negative when $z > \by$. Thus from Corollary \ref{signuncond} the covariance $\cov(f,g)$ is positive for any increasing symmetric functions $f(Y)$ and $g(Z)$, in particular for $f(Y) = Y^2$ and $g(Z) = Z^2$.

\end{document}